\documentclass[11pt]{amsart}
\usepackage{amsmath}
\usepackage{amsfonts}
\usepackage{amsthm}
\usepackage{amssymb}
\usepackage{verbatim}
\usepackage{enumerate}
\usepackage{graphicx}
\usepackage{longtable}
\usepackage[all,cmtip]{xy}
\usepackage{upref}
\usepackage{hyperref}
\usepackage[hang,small,bf]{caption}

\newtheorem{dfn}{Definition}[section]
\newtheorem{thm}[dfn]{Theorem}
\newtheorem{prp}[dfn]{Proposition}
\newtheorem{rmk}[dfn]{Remark}
\newtheorem{lem}[dfn]{Lemma}
\newtheorem{cor}[dfn]{Corollary}
\newtheorem*{qst}{Question}

\numberwithin{equation}{section}

\newcommand{\Z}{{\mathbb Z}}

\newcommand{\R}{{\mathbb R}}

\newcommand{\C}{{\mathbb C}}
\newcommand{\W}{{\mathcal W}}

\begin{document}

\title{Magnetic Katok examples on the two-sphere}

\author{Gabriele Benedetti}
\address{WWU M\"unster, Mathematisches Institut, Einsteinstra\ss e 62, 48149 M\"unster, Germany}
\email{\href{mailto:benedett@uni-muenster.de}{benedett@uni-muenster.de}}
\date{\today}

\begin{abstract}
We show that there exist non-exact magnetic flows on the two-sphere having an energy level whose double cover is strictly contactomorphic to an irrational ellipsoid in $\C^2$. Our construction generalises the examples of integrable Finsler metrics on the two-sphere with only two closed geodesics due to A.~Katok. 
\end{abstract}

\maketitle

\section{Introduction}
Let $\sigma$ be a two-form on $S^2$. We endow $T^*S^2$ with the symplectic form \begin{equation*}
\omega_{\sigma}\,:=\ d\lambda\ -\ \pi^*\sigma\,,
\end{equation*}
where $\lambda$ is the canonical $1$-form on $T^*S^2$ and $\pi$ is the bundle projection $\pi:T^*S^2\rightarrow S^2$. If $H:U\rightarrow\R$ is a smooth function on an open subset $U$ of $T^*S^2$, we define its Hamiltonian vector field $X_{H,\sigma}$ by
\begin{equation*}
\imath_{X_{H,\sigma}}\omega_\sigma\ =\ -dH\,.
\end{equation*}
We denote by $\Phi^{H,\sigma}$ the flow generated by $X_{H,\sigma}$.

If $g$ is a Riemannian metric on $S^2$, a natural choice for a Hamiltonian is the kinetic energy function $H_g:T^*S^2\rightarrow\R$. It is defined as $H_g(q,p):=\frac{1}{2}|p|^2_q$, where $|\cdot|_q$ is the norm induced by $g$ on $T_q^*S^2$. In this case we refer to $\Phi^{g,\sigma}:=\Phi^{H_g,\sigma}$ as the \textit{magnetic flow} of the pair $(g,\sigma)$. The reason for such a terminology is that these dynamical systems bear an interesting physical significance: they describe the motion of a charged particle on $S^2$ under the influence of a stationary magnetic field \cite{Arn61}. In the following discussion, we will also call $\sigma$ a \textit{magnetic form} and the pair $(g,\sigma)$ a \textit{magnetic system}.

A classical problem is to study the existence of periodic orbits for a magnetic flow $\Phi^{g,\sigma}$ on the energy level $H_g^{-1}(k)$, as $k$ varies in $(0,+\infty)$. When the level is of \textit{contact type} in $(T^*S^2,\omega_{\sigma})$ (see \cite[Chapter 4.3]{HZ11}), it is known that the number of periodic orbits is at least two \cite{CGH12,GHHM13}. If the energy level is in addition \textit{dynamically convex} in $(T^*S^2,\omega_\sigma)$, then such an existence result can be refined to yield the following dichotomy: either the number of periodic orbits is exactly $2$ or it is infinite. We refer the reader to \cite{HWZ98,HLS15} for the definition of dynamical convexity and a proof of the dichotomy.

In \cite{Ben16} we proved that low energy levels are dynamically convex provided the magnetic form is symplectic on $S^2$. We recall the precise statement here for the convenience of the reader. Before doing that, we remark that this result can be thought as a generalisation of a theorem by Harris and G.~Paternain asserting that suitably pinched Finsler metrics have dynamically convex geodesic flows \cite{HP08}.
\begin{thm}[\cite{Ben16}]\label{thm:dc}
Let $(g_*,\sigma_*)$ be a magnetic system on $S^2$ such that $\sigma_*$ is symplectic. There exists a neighbourhood $\mathcal U$ of $(g_*,\sigma_*)$ in the $C^1$-topology and a positive real number $c_{\mathop{dc}}(g_*,\sigma_*)$ such that for every $(g,\sigma)\in \mathcal U$ and every $k<c_{\mathop{dc}}(g_*,\sigma_*)$, the hypersurface $H^{-1}_g(k)$ is dynamically convex in $(T^*S^2,\omega_\sigma)$.
\end{thm}
It is interesting to understand better the dichotomy between two and infinitely many periodic orbits in the energy regime given by the theorem above. For example, we can pose the following question, whose answer is still open.
\begin{qst}
Given a magnetic system $(g_*,\sigma_*)$ with $\sigma_*$ symplectic, does there exists a positive number $c_\infty(g_*,\sigma_*)$ such that, for every $k<c_\infty(g_*,\sigma_*)$, $\Phi^{g_*,\sigma_*}$ has infinitely many periodic orbits with energy $k$?
\end{qst}

Results of Franks and Handel \cite{FH03} and of Ginzburg, G\"urel and Macarini \cite{GGM15} show that for symplectic magnetic systems on orientable surfaces of \textit{positive} genus, the answer to this question is positive. Moreover, this holds also on the two-sphere if the magnetic system satisfies some non-resonance condition, as some work in progress by the author will show. In both cases the value of $c_\infty(g_*,\sigma_*)$ can be required to be uniform in a $C^1$-neighbourhood of $(g_*,\sigma_*)$.

In this paper we show that under this additional uniformity assumption the answer to the question above is negative. To this purpose let $\bar{g}$ be a metric on $S^2$ with constant curvature $1$ and let $\bar{\mu}$ be the area form associated to $\bar{g}$ and a given orientation on $S^2$. Moreover, for every $\alpha\in (-1,1)$ let \begin{equation*}
E_{\alpha}\,:=\, \left\{\,(z_1,z_2)\in\C^2\ \Big|\ (1+\alpha)\frac{|z_1|^2}{2}+(1-\alpha)\frac{|z_2|^2}{2}=1\,\right\}\ \subset\ \big(\C^2,d\lambda_{\C^2}\big) 
\end{equation*}
be the standard ellipsoid, where $\lambda_{\C^2}$ is the standard Liouville form in $\C^2$. It is a classical fact that $E_\alpha$ is a dynamically convex hypersurface and, when $\alpha$ is irrational, it carries only two periodic orbits \cite{HWZ98}.

\begin{thm}\label{thm:main}
For all $s\geq 0$, there exist sequences of Riemannian metrics $(g_n)$, $1$-forms $(\beta_n)$ on $S^2$ and positive numbers $(k_n)$ (depending on $s$) with the property that
\begin{equation}
g_n\ \xrightarrow{\, C^\infty\, }\ \bar{g}\,,\quad \beta_n\ \xrightarrow{\, C^\infty\, }\ 0\,,\quad k_n\ \longrightarrow\ 0
\end{equation}
and such that the double cover of the level $H^{-1}_{g_n}(k_n)\subset (T^*S^2,\omega_{s\bar{\mu}+d\beta_n})$ is strictly contactomorphic to $E_{\alpha_n}$, where $(\alpha_n)\subset(0,1)$ is a sequence of irrational numbers converging to $0$. In particular, $\Phi^{g_n,s\bar{\mu}+d\beta_n}$ has exactly two periodic orbits with energy $k_n$.
\end{thm}
\begin{rmk}
It would be interesting to generalise Theorem \ref{thm:main} to higher dimensional complex projective spaces. If $\ell$ is a positive integer, we can consider $\C\mathbb P^\ell$ endowed with K\"ahler metric $g_{\C\mathbb P^\ell}$ and K\"ahler form $\omega_{\C\mathbb P^\ell}$ (observe that $g_{\C\mathbb P^1}=\bar{g}$ and $\omega_{\C\mathbb P^1}=\bar{\mu}$). In \cite[Section 7]{GG04} a hypersurface $\Sigma\subset(T^*\C\mathbb P^\ell,d\lambda-s\pi^*\omega_{\C\mathbb P^\ell})$ with the following properties is constructed:
\begin{itemize}
\item $\Sigma$ can be taken arbitrarily $C^\infty$-close to an energy level of $H_{g_{\C\mathbb P^\ell}}$;
\item every Hamiltonian flow which has $\Sigma$ as a regular energy level possesses exactly $\ell(\ell+1)$ periodic orbits on such hypersurface.
\end{itemize}
Then, one could ask if the dynamics on $\Sigma$ can be realised (up to time reparametrisation) by an energy level of some magnetic system on $\C\mathbb P^\ell$. 
\end{rmk}

As mentioned above, we are not able to construct the sequence of energy levels having only two periodic orbits inside a fixed magnetic system, but we need to consider different systems for different values of the energy. The reason for that is that the set of Hamiltonian functions that can be written as the kinetic energy of a Riemannian metric is somewhat too small for our purposes. In view of this observation, we move to consider a more general class of Hamiltonians, which we denote by $\W(S^2)$. We say that a smooth function $H:U_H\to\R$ belongs to $\W(S^2)$ if: i) $U_H\subset T^*S^2$ is a fibrewise star-shaped neighbourhood of the zero section; ii) $H$ is strictly convex in the fibres; iii) the value $0\in\R$ is a global minimum for $H$ and it is achieved exactly at the zero section (namely, $H(q,p)=0$ if and only if $p=0$).

When $H_*\in\W(S^2)$ (or, more generally, $S^2$ is replaced by some closed manifold) and $\sigma_*$ is symplectic, the Hamiltonian flow of $(H_*,\sigma_*)$ was studied by Ginzburg and G\"urel \cite{GG09} in order to generalize the Weinstein--Moser theorem \cite{Wei73,Mos76,Mos78} and obtain the existence of periodic orbits with low energy. In this setting, the relevant observation for us is that the vertical Hessian of $H_*$ at the zero section yields a Riemannian metric $g_*=g(H_*)$ on $S^2$. This means that, up to shrinking $U$, there is some $C>0$ such that
\[
\big\|d_{(q,p)}(H_*-H_{g_*})\big\|\ \leq\ C\,|p|_q^2\,.
\]
This estimate implies that the arguments of \cite{Ben16} still apply to the Hamiltonian flow of $(H_*,\sigma_*)$ and yield a generalisation of Theorem \ref{thm:dc} to this kind of systems. In particular,  $H_*^{-1}(k)$ is dynamically convex in $(T^*S^2,\omega_{\sigma_*})$ if $k$ is small enough. The corresponding version of the question we formulated above now has a neater negative answer.
\begin{thm}\label{thm:main2}
For all $s>0$, there exists a function $H_s\in\W(S^2)$ such that, for $k$ small enough, the double cover of the level $H^{-1}_{s}(k)\subset (T^*S^2,\omega_{s\bar{\mu}})$ is strictly contactomorphic to the ellipsoid $E_{\varepsilon_s k}$, for some $\varepsilon_s>0$. In particular, $\Phi^{H_s,s\bar{\mu}}$ has exactly two periodic orbits whenever $\varepsilon_s k$ is irrational.
\end{thm} 

Theorem \ref{thm:main} and \ref{thm:main2} yield energy levels whose dynamics is conjugated up to time reparametrisation to the Finsler flow of the celebrated \textit{Katok examples}. This is a class of Finsler metrics with only two closed geodesics and integrable flow. They were introduced by A. Katok in \cite{Kat73} and studied further by Ziller in \cite{Zil83}. The double cover of the sphere bundle of the Katok examples was proven to be strictly contactomorphic to an ellipsoid of the type above by Harris and G. Paternain in \cite[Section 4 and 5]{HP08}. Their result yields the case $s=0$ of the theorem above, since it was also known that such a sphere bundle is contactomorphic to an exact magnetic system on $S^2$ (see \cite{Pat99}). In order to construct the contactomorphism for $s>0$, which is the case of interest in view of the question we asked, we will reduce to the work of Harris and G. Paternain by the means of the following proposition.

\begin{prp}\label{prp:sym}
There exists a family of symplectomorphisms
\begin{equation}
\Psi_s:\Big(\big\{|p|>s\big\},d\lambda\Big)\ \longrightarrow\ \Big(\big\{|p|>0\big\},d\lambda\,-\,s\pi^*\bar{\mu}\Big)\,,\quad s\in[0,+\infty)\,,
\end{equation}
which is equivariant with respect to orientation-preserving isometries of $\bar{g}$.
The map $\Psi_0$ is the identity and, for every $s>0$, $\Psi_s$ admits a continuous extension $\Psi_s:\big\{|p|\geq s\big\}\rightarrow T^*S^2$ sending $\big\{|p|=s\big\}$ onto $\big\{|p|=0\big\}$.
\end{prp}
\begin{rmk}
Let us consider the contact manifold
\begin{equation*}
(\R\mathbb P^3,\lambda_{\R\mathbb P^3})\,:=\ \big(\{|p|=1\},\lambda|_{\{|p|=1\}}\big)
\end{equation*}
and let $\big((0,+\infty)\times \R\mathbb P^3,d(R\lambda_{\R\mathbb P^3})\big)$ be its symplectisation. It is known that $\big(\{|p|>0\},d\lambda\big)$ is symplectomorphic to such symplectisation. Then, Proposition \ref{prp:sym} implies that there is a symplectomorphism
\begin{equation*}
\chi_s:\big(\{|p|>0\},d\lambda-s\pi^*\bar{\mu}\big)\longrightarrow\big((s,+\infty)\times \R\mathbb P^3,d(R\lambda_{\R\mathbb P^3})\big)
\end{equation*}
and $\Psi_s$ fits into the commutative diagram
\begin{equation}\label{eq:dia}
\xymatrix{ \Big(\big\{|p|>0\big\},d\lambda\,-\,s\pi^*\bar{\mu}\Big)\ar[r]^-{\imath_1\, \circ\, \Psi^{-1}_s} \ar[d]_{\chi_s}& \Big(\big\{|p|>0\big\},d\lambda\Big)\ar[d]^{\chi_0} \\
 \Big((s,+\infty)\times \R\mathbb P^3,d(R\lambda_{\R\mathbb P^3})\Big)\ar[r]^{\imath_2} & \Big((0,+\infty)\times \R\mathbb P^3,d(R\lambda_{\R\mathbb P^3})\Big) \,,}
\end{equation}
where $\imath_1$ and $\imath_2$ are natural inclusions.
\end{rmk}

We end the introduction with a brief summary of the paper. Proposition \ref{prp:sym} will be proven in Section \ref{sec:sym}. Then, in Section \ref{sec:kat} we will show that the image of the sphere bundle of the Katok examples under $\Psi_s$ is contactomorphic to an energy level of a Hamiltonian, whose flow is a superposition of two commuting circle actions on the cotangent bundle. By suitably modifying such a Hamiltonian we construct examples satisfying the hypotheses of Theorem \ref{thm:main} and of Theorem \ref{thm:main2}, respectively.

\section*{Acknowledgements}
We are grateful to Gabriel Paternain for drawing our attention on the question and to Peter Albers for suggesting the commutative diagram \eqref{eq:dia}. We also thank Viktor Ginzburg and Leonardo Macarini for useful discussions. The author was supported by the DFG grant ``SFB 878 - Groups, Geometry, and Actions".

\section{The family of symplectomorphisms}\label{sec:sym}
We begin with some preliminaries from the geometry of the tangent bundle of the two-sphere (see \cite{GK02} for the proofs and more details).
\bigskip

For every $(q,v)\in TS^2$ we have linear maps $L^\mathcal H_{(q,v)}:T_qS^2\rightarrow T_{(q,v)}TS^2$ and $L^{\mathcal V}_{(q,v)}T_qS^2\rightarrow T_{(q,v)}TS^2$. The former is the \textit{horizontal lift} associated to the Levi-Civita connection of $\bar{g}$, while the latter is the \textit{vertical lift} induced by the inclusion of the vertical fibers $T_qS^2\hookrightarrow TS^2$. If $\mathcal H$ and $\mathcal V$ are the distributions in $T(TS^2)$ given by the images of these maps, we have
\begin{equation}
\mathcal H\oplus\mathcal V\ =\ T(TS^2)\,. 
\end{equation}
Such a splitting yields also two projection maps $P^\mathcal V_{(q,v)}:T_{(q,v)}TS^2\rightarrow T_qS^2$ and $P^\mathcal H_{(q,v)}:T_{(q,v)}TS^2\rightarrow T_qS^2$ with the property that
\begin{align}
\ker P^\mathcal V_{(q,v)}\ =\ \mathcal V\,,&\quad P^\mathcal V_{(q,v)}\circ L^\mathcal H_{(q,v)}\ =\ \mathop{Id}_{T_qS^2}\,,\\
\ker P^\mathcal H_{(q,v)}\ =\ \mathcal H\,,&\quad P^\mathcal H_{(q,v)}\circ L^\mathcal V_{(q,v)}\ =\ \mathop{Id}_{T_qS^2}\,.
\end{align}
We have $P^{\mathcal V}=d\pi$, while $P^{\mathcal H}$ is called the \textit{connection map}.

Let us denote by $\jmath:TS^2\rightarrow TS^2$ the fibrewise rotation by an angle of $\pi/2$. We define the following four vector fields on $TS^2$
\begin{equation}
\left\{
\begin{aligned}
X_{(q,v)}\,&:=\ L^\mathcal H_{(q,v)}(v)\,,& Y_{(q,v)}\,:=& \ L^{\mathcal V}_{(q,v)}(v)\,,\\
H_{(q,v)}\,&:=\ L^\mathcal H_{(q,v)}(\jmath_qv)\,,& V_{(q,v)}\,:=& \ L^{\mathcal V}_{(q,v)}(\jmath_qv)\,.
\end{aligned}
\right.
\end{equation}
In a dual fashion, we define four $1$-forms on $TS^2$
\begin{equation}
\left\{
\begin{aligned}
\zeta^X_{(q,v)}\,&:=\ \bar{g}_q\big(P^\mathcal V_{(q,v)}(\cdot),v\big)\,,&\zeta^Y_{(q,v)}\,&:=\ \bar{g}_q\big(P^\mathcal H_{(q,v)}(\cdot),v\big)\,,\\
\zeta^H_{(q,v)}\,&:=\ \bar{g}_q\big(P^\mathcal V_{(q,v)}(\cdot),\jmath_qv\big)\,,& \zeta^V_{(q,v)}\,&:=\ \bar{g}_q\big(P^\mathcal H_{(q,v)}(\cdot),\jmath_qv\big)\,.
\end{aligned}
\right.
\end{equation}
Using the identification $TS^2\rightarrow T^*S^2$ given by the metric $\bar{g}$, we can push forward the vector fields and the $1$-forms we have just defined to the cotangent bundle of the two-sphere. By abusing notation we will use the same names for the new objects. We have the relations
\begin{equation}
\lambda\ =\ \zeta^X\,,\quad\quad \pi^*\sigma\ =\ d\left(-\frac{\zeta^V}{|p|^2}\right)\,.
\end{equation}
In particular, on the complement of the zero section we have
\begin{equation}
\omega_{s\bar{\mu}}\ =\ d\lambda_s\,,\quad\quad \lambda_s\,:=\ \zeta^X\,+\,s\frac{\zeta^V}{|p|^2}\,.
\end{equation}

Let us consider now the flow of the vector fields $Y$ and $H':=\frac{H}{|p|}$. If $b>0$, then we have $\Phi^Y_{\log b}(q,p)=(q,bp)$. On the other hand, the trajectory $t\mapsto\Phi^{H'}_t(q,v)$ is given by the curve $(\gamma,-|v|\jmath_\gamma \dot\gamma)$, where $\gamma:\R\rightarrow S^2$ is a unit speed geodesic such that $\gamma(0)=q$ and $\dot\gamma(0)=v/|v|$. There holds
\begin{equation}
\left\{\begin{aligned}
(\Phi^{Y}_{\log b})^*\zeta^X\ &=\ b\,\zeta^X\,,&\quad (\Phi^{H'}_a)^*\zeta^X\ &=\ \cos a\,\zeta^X\,+\,\sin a\,\frac{\zeta^V}{|p|}\,,\\ 
(\Phi^{Y}_{\log b})^*\frac{\zeta^V}{|p|}\ &=\ b\,\frac{\zeta^V}{|p|}\,,&\quad (\Phi^{H'}_a)^*\frac{\zeta^V}{|p|}\ &=\ \cos a\,\frac{\zeta^V}{|p|}\,-\,\sin a\,\zeta^X\,.
\end{aligned}\right.
\end{equation}

In view of these formulae, we make the Ansatz
\begin{equation}
\Psi_s(q,p)\,:=\ \Phi^Y_{\log b_s(|p|)}\Big(\Phi^{H'}_{a_s(|p|)}(q,p)\Big)\,,
\end{equation}
where $a_s:(\rho_s,+\infty)\rightarrow\R$ and $b_s:(\rho_s,+\infty)\rightarrow(0,+\infty)$ are functions to be determined.
We compute
\begin{align*}
d\Psi_s\cdot \xi\ &=\ \Big(d\big(\log b_s\circ|p|\big)\cdot\xi\Big)Y_{\Psi_s}\ +\ \Big(d\big(a_s\circ|p|\big)\cdot\xi\Big)d_{\Phi^{H'}_{a_s}}\!\Phi^Y_{\log b_s}H'_{\Phi^{H'}_{a_s}}\\
&+\ d_{\Phi^{H'}_{a_s}}\!\Phi^Y_{\log b_s}d\Phi^{H'}_{a_s}\cdot\xi\,.
\end{align*}
Since $\zeta^X$ and $\zeta^V$ vanish on both $Y$ and $H'$ only the last term above contributes to $\Psi_s^*\lambda_s$:
\begin{align*}
\Psi_s^*\lambda_s\ &=\ \big(\Phi^{H'}_{a_s}\big)^*\Big(\big(\Phi^{Y}_{\log b_s}\big)^*\lambda_s\Big)\\
&=\ \big(\Phi^{H'}_{a_s}\big)^*\Big(b_s\zeta^X\,+\,\frac{s}{|p|}\frac{\zeta^V}{|p|}\Big)\\
&=\ b_s\Big(\cos a_s\,\zeta^X\,+\,\sin a_s\,\frac{\zeta^V}{|p|}\Big)\, +\, \frac{s}{|p|}\Big(\cos a_s\,\frac{\zeta^V}{|p|}\,-\,\sin a_s\,\zeta^X\Big)\\
&=\ \Big(b_s\cos a_s\,-\,\frac{s\sin a_s}{|p|}\Big)\zeta^X\, +\, \Big(b_s\sin a_s\,+\,\frac{s\cos a_s}{|p|}\Big)\frac{\zeta^V}{|p|}\,.
\end{align*}
Therefore, $\Psi_s^*\lambda_s=\lambda$ as soon as
\begin{equation}
\left\{\begin{aligned}
b_s\cos a_s\,-\,\frac{s\sin a_s}{|p|}\ &=\ 1\,,\\
b_s\sin a_s\,+\,\frac{s\cos a_s}{|p|}\ &=\ 0\,.
\end{aligned}\right.
\end{equation}
We express $b_s$ as a function of $a_s$ from the second equation and we obtain $b_s=-\frac{s}{|p|\tan a_s}$. Substituting $b_s$ in the first equation yields
\begin{equation}
\sin a_s\ =\ -\frac{s}{|p|}
\end{equation}
and we also get
\begin{equation}
b_s\ =\ \sqrt{1-\frac{s^2}{|p|^2}}\,.
\end{equation}
Therefore, we can take $\rho_s=s$ and the functions $a_s:(s,+\infty)\rightarrow(-\pi/2,0)$ and $b_s:(s,+\infty)\rightarrow(0,1)$ are both monotonically increasing bijections.

On the tangent bundle the map $\Psi_s$ has the following easy description. Let $(q,v)$ be such that $|v|>s$ and let $\Pi(q,v)\in S^2$ be the centre of the positive hemisphere bounded by the geodesic passing through $q$ with velocity $v$. Then, let $(\theta,\varphi)\in[0,\pi]\times\R/2\pi\Z$ be positively oriented geodesic polar coordinates centered at $\Pi(q,v)$ and consider induced coordinates $(\theta,\varphi,v_\theta,v_\varphi)$ on $TS^2$. 
We have the identification $(q,v)=\big(\pi/2,\varphi(q),0,|v|\big)$, for some $\varphi(q)\in \R/2\pi\Z$. In these coordinates the following formula holds
\begin{equation*}
\Psi_s\big(\pi/2,\varphi(q),0,|v|\big)\ =\ \big(\theta(s),\varphi(q),0,|v|\big)\,,\quad \mbox{where}\ \cos\theta(s)\ =\ \frac{s}{|v|}\,. 
\end{equation*} 
In view of this description we also see that we can extend $\Psi_s$ by setting $\Psi_s(q,v)=(\Pi(q,v),0)$, if $|v|=s$. 

We now observe that $\Psi_s$ is equivariant with respect to any orientation preserving isometry $I:S^2\rightarrow S^2$. Consider the maps
\begin{align}
\check{I}:&TS^2\ \longrightarrow\ TS^2\,,&\ \check I(q,v)\,&:=\ \big(I(q),d_qI(v)\big)\,,\\
\hat I:&T^*S^2\ \longrightarrow\ T^*S^2\,,&\ \hat I(q,p)\,&:=\ \big(I(q),p\circ (d_qI)^{-1}\big)\,.\label{eq:ind}
\end{align}
Since $I$ is an isometry, $\check I$ and $\hat I$ are conjugated by the metric duality. Therefore, it is enough to prove that for every $a\in\R$ and $b\in (0,+\infty)$ we have
\begin{equation}
\Phi^Y_{\log b}\circ \check I\ =\ \check I\circ \Phi^Y_{\log b}\,,\quad\quad \Phi^{H'}_a\circ \check I\ =\ \check I\circ \Phi^{H'}_a\,.
\end{equation}
For the first identity we compute
\begin{equation*}
\check I\big(\Phi^Y_{\log b}(q,v)\big) = \check I(q,bv) = \big(I(q),d_qI\cdot bv\big) = \big(I(q),b\,d_qI\cdot v\big) = \Phi^Y_{\log b}\big(\check I(q,v)\big)\,.
\end{equation*}
For the second identity we preliminarily observe that since $I$ is orientation preserving, we have $\check I\circ \jmath=\jmath\circ \check I$. Then,
\begin{align*}
\check I\big(\Phi^{H'}_a(q,v)\big) = \check I\big(\gamma(a),-|v|\jmath_{\gamma(a)}\dot\gamma(a)\big)& = \big(I(\gamma(a)),-|v|d_{\gamma(a)} I\cdot\jmath_{\gamma(a)}\dot\gamma(a)\big)\\
& = \big((I\circ\gamma)(a),-|v|\jmath_{(I\circ\gamma)(a)}\dot{(I\circ\gamma)}(a)\big)\,.
\end{align*}
The conclusion follows as $I\circ\gamma$ is still a geodesic. The proof of Proposition \ref{prp:sym} is now complete.

\section{The construction of the examples}\label{sec:kat}
In this section we construct examples satisfying the hypotheses of Theorem \ref{thm:main} and Theorem \ref{thm:main2}, respectively. As a first common step, we see how the two standard Hamiltonian circle actions on $(T^*S^2,\omega_0)$, the one given by the geodesic flow and the one given by the rotation around an axis, behave under the map $\Psi_s$ constructed in the previous section. For every $\alpha\\in[0,1)$, this will enable us to construct a family of Hamiltonians $H_{s,\alpha}:(T^*S^2,\omega_{s\bar\mu})\to\R$ whose dynamics is conjugated, up to time reparametrisation, to the geodesic flow of the Katok example with parameter $\alpha$. By suitably modifying the functions $H_{s,\alpha}$ we will get the desired examples.
\bigskip

Let $(\theta,\varphi)\in[0,\pi]\times\R/2\pi\Z$ be some positively oriented geodesic polar coordinates for the metric $\bar{g}$. In these coordinates the metric $\bar{g}$ and the form $\bar{\mu}$ have the expressions
\begin{equation}
\bar{g}\ =\ d\theta^2\ +\ (\sin\theta)^2 d\varphi^2\,,\quad\quad\bar{\mu}\ =\ \sin\theta\, d\theta\wedge d\varphi\,.
\end{equation}
We consider the vector field $\partial_\varphi\in\Gamma(S^2)$. It is such that
\begin{enumerate}
\item it generates a flow of isometries $t\mapsto\Phi^{\partial_\varphi}_t$;
\item the $1$-form $\imath_{\partial_\varphi}\bar{\mu}$ is exact. Indeed,
\begin{equation*}
\imath_{\partial_\varphi}\big(\sin\theta\, d\theta\wedge d\varphi\big)\ =\ -\sin\theta\, d\theta\ =\ d(\cos\theta)\,.
\end{equation*}
\end{enumerate}
We take $h:=\cos\theta$ as preferred primitive for $\imath_{\partial_\varphi}\bar{\mu}$. Moreover, we write $\beta\in\Omega^1(S^2)$ for the $1$-form dual to $\partial_\varphi$ with respect to $\bar{g}$. In coordinates it has the expression $\beta=(\sin\theta)^2\, d\varphi$. Finally, let $t\mapsto{\hat{\Phi}}^{\partial_\varphi}_t$ be the flow on $T^*S^2$ induced by the differential of $\Phi^{\partial_\varphi}$ as prescribed by formula \eqref{eq:ind}.
The flow $\hat\Phi^{\partial_\varphi}$ induces a free $\R/2\pi\Z$-action on the complement of the zero section and it is Hamiltonian with respect to $\omega_{s\bar{\mu}}$. This can be seen as a generalisation of Exercise 4.2A in \cite{AM78}, which deals with exact magnetic forms.
\begin{lem}
The flow $\hat\Phi^{\partial_\varphi}$ is generated by the Hamiltonian vector field $X_{\Omega_s,s\bar{\mu}}$, where
\begin{equation}
\Omega_s(q,p)\,:=\ p\big((\partial_\varphi)_q\big)\ +\ s\,h(q)\,.
\end{equation}
\end{lem}
\begin{proof}
By Corollary 4.2.11 in \cite{AM78}, $\hat\Phi^{\partial_\varphi}$ is a Hamiltonian flow with respect to the symplectic form $\omega_0$ and with associated function $p\big((\partial_\varphi)_q\big)$. Therefore, the thesis follows by computing
\begin{equation*}
\imath_{\left(\frac{d}{dt}\hat\Phi^{\partial_\varphi}_t\right)}\pi^*\sigma\ =\ \imath_{\partial_\varphi}\sigma\ =\ dh\,.\qedhere
\end{equation*}
\end{proof}

Let us define $R_s:T^*S^2\rightarrow\R$ by $R_s(q,p):=\sqrt{|p|^2+s^2}$. The importance of this function is clarified by the following lemma.
\begin{lem}
For every $s\geq0$, we have
\begin{equation}
R_s\circ \Psi_s\ =\ R_0\,,\quad\quad \Omega_s\circ \Psi_s\ =\ \Omega_0\,,\quad\mbox{on  }\ \big\{|p|\geq s\big\}\,.
\end{equation} 
\end{lem}
\begin{proof}
From the formula for $\Psi_s$ we know that
\begin{equation*}
|\Psi_s(q,p)|\ =\ |b_s(|p|)p|\ =\ \sqrt{|p|^2-s^2}\quad\Longrightarrow\quad \sqrt{|\Psi_s(q,p)|^2+s^2}\ =\ |p|\,.
\end{equation*}
For the second equality we argue indirectly. Since $\Psi_s$ is $\hat\Phi^{\partial_\varphi}_t$ equivariant for every $t$, we know that $d\Psi_s\cdot X_{\Omega_0,0}=X_{\Omega_s,s\bar{\mu}}$. Therefore, $\Omega_s\circ \Psi_s$ and $\Omega_0$ have the same $\omega_0$-Hamiltonian flow. This means that there exists a constant $c(s)$ such that $\Omega_s\circ \Psi_s=\Omega_0+c(s)$. In order to find the value of $c(s)$, we simply study the image of $\{|p|=s\}$ under $\Omega_s\circ \Psi_s$ and $\Omega_0$:
\begin{align*}
\Omega_s\circ \Psi_s\Big(\big\{|p|=s\big\}\Big)\ &=\ \Omega_s\Big(\big\{|p|=0\big\}\Big)\ =\ s\,h(S^2)\ =\ [-s,s]\,,\\
\Omega_0\Big(\big\{|p|=s\big\}\Big)\ &=\ \big[-s\,\sup|\partial_\varphi|,s\,\sup|\partial_\varphi|\,\big]\ =\ [-s,s]\,.
\end{align*}
Therefore $c(s)=0$.
\end{proof}
Let $\alpha\in[0,+\infty)$ and define the function
\begin{equation}
H_{s,\alpha}\,:=\ R_s\ +\ \alpha\,\Omega_s\,. 
\end{equation}
In view of the previous lemma, we have $H_{s,\alpha}\circ \Psi_s= H_{0,\alpha}$ on $\{|p|\geq s\}$. Let us analyse now the properties of $H_{s,\alpha}$.
\begin{lem}\label{lem:has}
The function $H_{s,\alpha}$ is fibrewise strictly convex. It is fibrewise coercive if and only if $\alpha\in[0,1)$. A level set $\{H_{0,\alpha}=c\}$ is contained in $\{|p|>s\}$ (or, equivalently, a level set $\{H_{s,\alpha}=c\}$ is contained in $\{|p|>0\}$) if and only if $c>s(1+\alpha)$.
\end{lem}
\begin{proof}
We observe that $R_s$ is strictly convex in the fibers since it is the composition of the fibrewise convex function $(q,p)\mapsto |p|$ and of the real function $m\mapsto \sqrt{m^2+s^2}$, which is convex and strictly increasing for $m\geq0$. The function $\alpha \Omega_s$ is also convex in the fibers as it is affine. Hence, the first statement follows since the sum of two convex functions is still convex. 

It is enough to prove the second statement for $s=0$, since $\Psi_s$ is a coercive map. In this case we see that, for every $m>0$,
\begin{equation}
H_{0,\alpha}\Big(\big\{|p|=m\big \}\Big)\ =\ \big[m(1-\alpha),m(1+\alpha)\big]\,.
\end{equation}
The conclusion follows by noticing that
\begin{equation}
\lim_{m\rightarrow+\infty}m(1-\alpha)\ =\ +\infty\quad \Longleftrightarrow\quad \alpha\in[0,1)\,.
\end{equation}
For the third statement we observe that
\begin{equation}
\big\{H_{0,\alpha}=c\big\}\ \subset\ \big\{|p|>s\big\}\quad\Longleftrightarrow\quad c\ >\ \sup_{|p|\leq s} H_{0,\alpha}\ =\ s(1+\alpha)\,.\qedhere
\end{equation}
\end{proof}

\begin{cor}\label{cor:iso1}
If $\alpha\in[0,1)$ and $c>s(1+\alpha)$, then the double cover of the Hamiltonian flow of $H_{s,\alpha}$ at level $c$ is conjugated with the Reeb flow of the ellipsoid $E_\alpha$.
\end{cor}
\begin{proof}
We know that $\Psi_s^*\omega_{s\bar{\mu}}=\omega_0$ and $H_{s,\alpha}\circ \Psi_s=H_{0,\alpha}$. By Lemma \ref{lem:has}, the condition $c>s(1+\alpha)$ ensures that it is enough to prove the corollary assuming $s=0$. In this case the statement holds by \cite[Section 5]{HP08}. 
\end{proof}
We are now in position to construct the examples in the class $\W(S^2)$. \begin{proof}[Proof of Theorem \ref{thm:main2}]
Let $s>0$ be given and consider some real number $\varepsilon>0$. We can define the function
\[
H_s:\big\{|p|<\delta_{s,\varepsilon}\big\}\ \longrightarrow\ \R\,,\qquad H_s\,:=\ \frac{R_s-s}{1+\varepsilon(s-\Omega_s)}\,,
\]
where $\delta_{s,\varepsilon}^2:=1/\varepsilon^2+2s/\varepsilon$. We claim that $H_s\in\W(S^2)$, up to shrinking the domain of definition. The claim follows by computing the 2-jet of $H_s$ at the zero section. For the computation we use that $H_s$ is obtained by multiplying the function $R_s-s$, which lies in $\W(S^2)$, with the positive function $\big(1+\varepsilon(s-\Omega_s)\big)^{-1}$. Indeed, for all $q\in S^2$ we have that $H_s(q,0)=0$, $d_{(q,0)}H_s=0$ and the vertical Hessian is
\[
g(H_s)_{(q,0)}\ =\ \frac{g(H)_{(q,0)}}{1+\varepsilon(s-\Omega_s(q,0))}\ =\ \frac{\bar g_q}{s\big(1+\varepsilon s(1-\cos\theta(q))\big)}\,.
\]
Since $g(H_s)_{(q,0)}$ is positive definite, the claim follows. We now make a second claim: if $k>0$ is small enough, then 
\[
\big\{H_s=k\big\}\ =\ \big\{ H_{s,\varepsilon k}=s(1+\varepsilon k)+k\big\}\,.
\]
Since $s(1+\varepsilon k)+k>s(1+\varepsilon k)$, if the claim is true, we can apply Corollary \ref{cor:iso1} with $\alpha=\varepsilon k$, and conclude that the hypersurface $\{H_s=k\}$ is contactomorphic to $E_{\varepsilon k}$, as we wanted to show. So let us prove the claim. First, observe that $\big\{H_s=k\big\}$ is a fibrewise convex hypersurface containing the zero section in its interior since $H_s\in\W(S^2)$. Then, by rearranging terms we see that if an element $(q,p)\in T^*S^2$ satisfies $H_s(q,p)=k$, we have
\[
k(1+\varepsilon s)+s\ =\ R_s(q,p)+\varepsilon k\Omega_s(q,p)\ =\ H_{s,\varepsilon k}(q,p)\,.
\]
Thus, there holds $\{H_s=k\}\subseteq\{ H_{s,\varepsilon k}=s(1+\varepsilon k)+k\}$. However, since $s(1+\varepsilon k)+k>s(1+\varepsilon k)$ also the set $\{ H_{s,\varepsilon k}=s(1+\varepsilon k)+k\}$ is fibrewise convex and contains the zero section in the interior. Thus, the two level sets must be equal and the claim is proved.
\end{proof}

We now move to construct the sequence of magnetic systems contained in Theorem \ref{thm:main}. The first task is to find a Finsler norm $F=F_{s,\alpha,k}$ and a $1$-form $\eta=\eta_{s,\alpha,k}$ on $S^2$, depending on $s\geq0$, $\alpha<1$ and $k>0$, such that
\begin{equation}\label{eq:fin}
\big\{\,(q,p)\ \big|\ H_{s,\alpha}(q,p-\eta_q)\,=\,c\,\big\}\ =\ \big\{\,(q,p)\ \big|\ F_q(p)\,=\,1\,\}\,,
\end{equation}
where $c=c_{s,\alpha,k}:=\sqrt{2k+s^2}+\alpha s>s(1+\alpha)$. We aim to choose $\eta$ in such a way that $F$ is the norm of a Riemannian metric. First, we see that \eqref{eq:fin} is equivalent to asking that for any $(q,p)$ in the complement of the zero section
\begin{equation}\label{eq:lev}
H_{s,\alpha}\left(q,\frac{p}{F_q(p)}-\eta_q\right)\ =\ c\,.
\end{equation}
We expand this expression (dropping the explicit dependence on $q$ from the notation) and we get
\begin{equation}\label{eq:root}
\sqrt{|p|^2-2\bar{g}(p,\eta)F(p)+(|\eta|^2+s^2)F(p)^2}\ =\ \big(c+\alpha\eta(\partial_\varphi)-\alpha sh\big)F(p) - \alpha p(\partial_\varphi)\,.
\end{equation}
We now square both sides and we arrive at the quadratic equation
\begin{equation}\label{eq:quad}
y_2F(p)^2\,+\,2y_1F(p)\,-\,y_0\ =\ 0\,.
\end{equation}
where the coefficients are given by
\begin{equation*}
\left\{\begin{aligned}
y_0\ &=\ |p|^2\,-\,\alpha^2\bar{g}(p,\beta)^2\,,\\
y_1\ &=\ \bar{g}(p,\eta)\,-\,\alpha\big(c+\alpha\eta(\partial_\varphi)-\alpha sh\big)\,\bar{g}(p,\beta)\,,\\
y_2\ &=\ \big(c+\alpha\eta(\partial_\varphi)-\alpha sh\big)^2\,-\,|\eta|^2\,-\,s^2\,.
\end{aligned}\right.
\end{equation*}
Imposing $y_1=0$, we obtain the equation for $\eta$
\begin{equation*}
\bar{g}(p,\eta)\ =\ \alpha\big(c\,+\,\alpha\eta(\partial_\varphi)\,-\,\alpha sh\big)	\,\bar{g}(p,\beta)\,.
\end{equation*}
Since this has to hold for every $(q,p)$ with $p\neq0$, we find that $\eta=\alpha r\beta$, where $r=r_{s,\alpha,k}:S^2\rightarrow\R$ is given by
\begin{equation}
r\,:=\ \frac{c-\alpha sh}{1-\alpha^2|\partial_\varphi|^2}\,.
\end{equation}
Solving now \eqref{eq:quad} with the coefficients given by the $1$-form $\eta$ that we have just found, we see that $y_2>0$ and we get the Finsler norm
\begin{equation}\label{eq:norm}
F(p)^2\ =\ \frac{1}{(1-\alpha^2|\partial_\varphi|^2)r^2-s^2}\,\Big(|p|^2\,-\,\alpha^2p(\partial_\varphi)^2\Big)\,.
\end{equation}
This Finsler norm satisfies \eqref{eq:lev} since the right-hand side of \eqref{eq:root} is positive with our choice of $\eta$ and $F$. For a proof of the positivity of $y_2$ and of the right-hand side of \eqref{eq:root} and for the exact computations leading to $F$ we refer the reader to the Appendix. If $g=g_{s,\alpha,k}$ is the Riemannian metric associated to $\sqrt{2k}F$, we have that
\begin{equation}\label{eq:met}
g(p_1,p_2)\,:=\ \frac{2k}{(1-\alpha^2|\partial_\varphi|^2)r^2-s^2}\Big(\bar{g}(p_1,p_2)\,-\,\alpha^2p_1(\partial_\varphi)p_2(\partial_\varphi)\Big)\,
\end{equation}
and
\begin{equation*}
\big\{\,(q,p)\ \big|\ H_{s,\alpha}\big(q,p-\alpha r(q)\beta_q\big)\,=\,c\,\big\}\ =\ \big\{\,(q,p)\ \big|\ H_{g}(q,p)\,=\,k\,\big\}\,,
\end{equation*}
where $H_g$ is the kinetic energy associated to $g$. Using the diffeomorphism $(q,p)\mapsto (q,p-\alpha r(q)\beta_q)$ we can pull back the Hamiltonian flow of $H_{s,\alpha}$ at level $c$ to a genuine magnetic flow with respect to a shifted magnetic form.
\begin{prp}\label{prp:iso2}
For every $s\geq 0$, $\alpha<1$ and $k>0$, the magnetic flow of the pair $\big(g_{s,\alpha,k},s\bar{\mu}+d(\alpha r_{s,\alpha,k}\beta)\big)$ at level $k$ is conjugated up to time reparametrisation to the Hamiltonian flow of $H_{s,\alpha}$ at level $c_{s,\alpha,k}$.
\end{prp}
The proposition yields a proof of Theorem \ref{thm:main} by a careful choice of the parameter $\alpha$ and the energy $k$.
\begin{proof}[Proof of Theorem \ref{thm:main}]
Choose an arbitrary sequence of positive numbers $(k_n)$ tending to zero and let $(\alpha_n)$ be a sequence of irrational numbers in $(0,1)$ such that
\begin{equation}\label{eq:inf}
\lim_{n\rightarrow\infty}\frac{\alpha_n}{k_n}\ =\ 0\,.
\end{equation}
We remark that when $s=0$, it would have been enough to assume that $\alpha_n$ is infinitesimal instead of requiring the stronger condition \eqref{eq:inf}. Then we apply Proposition \ref{prp:iso2} and obtain a corresponding sequence of magnetic systems $\big(g_n,s\bar{\mu}+d(\alpha_nr_n\beta)\big)=\big(g_{s,\alpha_n,k_n},s\bar{\mu}+d(\alpha_nr_{s,\alpha_n,k_n}\beta)\big)$. The fact that the energy level $\{H_{g_n}=k_n\}\subset(T^*S^2,\omega_{s\bar{\mu}+d(\alpha_nr_n\beta)})$ is strictly contactomorphic to the ellipsoid $E_{\alpha_n}$ follows by Proposition \ref{prp:iso2} and Corollary \ref{cor:iso1}. Thus, we need only show that $g_n\rightarrow \bar{g}$ and $\alpha_nr_n\beta\rightarrow 0$ in the $C^\infty$-topology. The latter limit is established at once by observing that
\begin{equation}
\lim_{n\rightarrow\infty}r_n\ =\ s\,.
\end{equation}
Let us deal with the former limit. Remembering the expression \eqref{eq:met} for $g_n$, we have only to check that
\begin{equation}
\lim_{n\rightarrow\infty}\frac{(1-\alpha^2_n|\partial_\varphi|^2)r_n^2-s^2}{2k_n}\ =\ 1
\end{equation}
By repeatedly using \eqref{eq:inf}, we see that the limit on the left-hand side is equal to 
\begin{equation*}
\lim_{n\rightarrow\infty}\frac{c^2_n-s^2}{2k_n}\ =\ \lim_{n\rightarrow\infty}\frac{2k_n+s^2+\alpha_n^2+2\alpha_n\sqrt{2k_n+s^2}-s^2}{2k_n}\ =\ 1\,.\qedhere
\end{equation*}
\end{proof}

\section*{Appendix}
In this short appendix we prove that if $\eta=\alpha r\beta$, then $y_2$ and the right-hand side of \eqref{eq:root} are positive and that $F$ has the expression given by \eqref{eq:norm}. We start by observing that from the equation $y_1=0$ we get 
\begin{equation*}
c\ +\ \alpha\eta(\partial_\varphi)\,-\,\alpha sh\ =\ r\,.
\end{equation*}
This yields at once $y_2=(1-\alpha^2|\partial_\varphi|^2)r^2-s^2$ and, hence, formula \eqref{eq:norm}. In order to prove that $y_2>0$, we observe that $c-\alpha sh>s$ and, therefore,
\begin{align*}
y_2\ >\ (1-\alpha^2|\partial_\varphi|^2)\frac{s^2}{(1-\alpha^2|\partial_\varphi|^2)^2}\ -\ s^2\ =\ s^2\left(\frac{1}{1-\alpha^2|\partial_\varphi|^2}\,-\,1\right)\ \geq\ 0\,.
\end{align*}
Finally, to prove that the right-hand side in \eqref{eq:root} is positive, it is enough to show that
\begin{equation}\label{eq:rhs}
r^2F(p)^2\ >\ \alpha^2p(\partial_\varphi)^2\,.
\end{equation}
From the definition of $F(p)$ we have
\begin{equation*}
F(p)^2\ \geq\ \frac{1}{(1-\alpha^2|\partial_\varphi|^2)r^2}\,\Big(|p|^2\,-\,\alpha^2p(\partial_\varphi)^2\Big)\,,
\end{equation*}
Therefore, \eqref{eq:rhs} is implied by
\begin{equation*}
|p|^2\,-\,\alpha^2p(\partial_\varphi)^2\ >\ (1-\alpha^2|\partial_\varphi|^2)\alpha^2p(\partial_\varphi)^2\,.
\end{equation*}
Bringing all terms on the left-hand side and using that $p(\partial_\varphi)\leq|p||\partial_\varphi|$, we see that it is sufficient to show that
\begin{equation*}
|p|^2-(2-\alpha^2|\partial_\varphi|^2)\alpha^2|p|^2|\partial_\varphi|^2>0\,.
\end{equation*}
The left-hand side is equal to $|p|^2(1-\alpha^2|\partial_\varphi|^2)^2$. Since this quantity is positive, the argument is complete.

\bibliographystyle{amsalpha}
\bibliography{2K}
\end{document}